\documentclass[reqno]{amsproc}

\usepackage{graphicx}

\usepackage{hyperref}

\hypersetup{
colorlinks=true,
linkcolor=blue,
anchorcolor=blue,
citecolor=blue
}

\usepackage{amssymb}
\usepackage{amsfonts}
\usepackage{amsmath}
    
\newtheorem{theorem}{Theorem}[section]
\newtheorem{lemma}[theorem]{Lemma}

\theoremstyle{definition}
\newtheorem{definition}[theorem]{Definition}

\numberwithin{equation}{section}

\newcommand{\newcom}{\newcommand}

\newcommand{\beq}{\begin{equation}}
\newcommand{\eeq}{\end{equation}}
\newcom{\ben}{\begin{eqnarray}}
\newcom{\een}{\end{eqnarray}}
\newcom{\beno}{\begin{eqnarray*}}
\newcom{\eeno}{\end{eqnarray*}}
\newcom{\bali}{\begin{aligned}}
\newcom{\eali}{\end{aligned}}

\newcommand{\f}{\frac}
\newcommand{\p}{\partial}

\newcommand{\vu}{\mathbf{u}}

\newcommand{\vc}[1]{{\boldsymbol #1}}
\newcommand{\Div}{{\rm div}}
\newcommand{\Grad}{\nabla}
\newcommand{\dx}{{\rm d} x}
\newcommand{\dt}{{\rm d} t }
\newcommand{\ds}{{\rm d} s}
\newcommand{\dxdt}{\dx\dt}
\font\F=msbm10 scaled 1000
\newcommand{\R}{\mbox{\F R}}

\begin{document}

\title{Remarks on weak-strong uniqueness for two-fluid model}

\author{Yang Li}
\address{School of Mathematical Sciences,  Anhui University, Hefei 230601, People's Republic of China}
%\curraddr{Department of Mathematics and Statistics,
%Case Western Reserve University, Cleveland, Ohio 43403}
\email{lynjum@163.com}
\thanks{The research of Y. Li was supported by National Natural Science Foundation of China under grant number 12001003.}

\author{Ewelina Zatorska}
\address{Department of Mathematics, Imperial College London,
 London SW7 2AZ, United Kingdom}
\email{e.zatorska@imperial.ac.uk}
\thanks{The work of E.Z. was supported by the EPSRC Early Career Fellowship no. EP/V000586/1.}


%\date{January 1, 1994 and, in revised form, June 22, 1994.}



\begin{abstract}
This paper concerns with the compressible two-fluid model with algebraic pressure closure. We prove a conditional weak-strong uniqueness principle, meaning that a finite energy weak solution, with bounded densities, coincides with the classical solution on the lifespan of the latter emanating from the same initial data.
\end{abstract}

\maketitle

% Do not use \tableofcontents for creating the TOC
% as it will put in page numbers, which will not be valid
% in the bound volume.
% Write all section titles manually as below.

\section*{Contents}
1. Introduction

2. Weak-strong uniqueness principle

3. An auxiliary lemma

References

\section{Introduction}
We consider the compressible two-fluid model with algebraic pressure closure in the three-dimensional torus $\mathbb{T}^3$:
\begin{equation}\label{eq1}
\left\{\begin{aligned}
& \p_t R+\Div_x (R \vu)=0,\\
& \p_t Q+\Div_x (Q \vu)=0,\\
& \p_t [ (R+Q)\vu ]+\Div_x [(R+Q) \vu \otimes \vu ]+\Grad_x p(Z)=\mu \Delta \vu+ (\mu+\lambda)\Grad_x \Div_x \vu.\\
\end{aligned}\right.
\end{equation}
Here, $R$ and $Q$ are densities of two fluids; $\vu\in\R^3$ means the velocity field. $p=p(Z)$ stands for the scalar pressure, relating implicitly to $R,Q$ through
\begin{equation*}
\left\{\begin{aligned}
& Q=\left(1-\f{R}{Z}  \right)Z^{\gamma},\,\,\gamma=\gamma_{+}/ \gamma_{-},  \\
& R\leq Z.\\
\end{aligned}\right.
\end{equation*}
The model describes the motion of two immiscible compressible fluids sharing the same
velocity field and obeying the algebraic pressure closure. The derivation of (\ref{eq1}) may be found in Bresch et al. \cite{BMZ19}. We refer to the monographs \cite{BDGG,IsHi06} for more discussions on related models.

Since the problem under consideration is evolutionary, we supplement (\ref{eq1}) with the initial conditions:
\beq\label{eq2}
(R,Q,\vu)|_{t=0}=(R_0,Q_0,\vu_0).
\eeq

We now introduce the concept of finite energy weak solution.
\begin{definition}
$(R,Q,\vu)$ is said to be a finite energy weak solution to the problem (\ref{eq1})-(\ref{eq2}) in $(0,T)\times \mathbb{T}^3$ provided that
\begin{itemize}
\item {  \[
(R,Z)\in L^{\infty}(0,T; L^{\gamma_{+}}(\mathbb{T}^3)),\,\, Q \in L^{\infty}(0,T; L^{\gamma_{-}}(\mathbb{T}^3)),
\]
\[
\Grad_x \vu \in L^2(0,T;L^2(\mathbb{T}^3;\R^{3\times 3}));
\]
}

\item { the equation of continuity for $R$
\[
\int_0^T\int_{\mathbb{T}^3}
\left(R \p_t \phi +R \vu\cdot \Grad_x \phi \right)\dxdt +\int_{\mathbb{T}^3}R_{0}\phi(0,\cdot)\dx=0
\] \\
for any $\phi \in C_c^{\infty}([0,T)\times \mathbb{T}^3)$;
}

\item { the equation of continuity for $Q$
\[
\int_0^T\int_{\mathbb{T}^3}
\left(Q \p_t \phi +Q \vu\cdot \Grad_x \phi \right)\dxdt +\int_{\mathbb{T}^3}Q_{0}\phi(0,\cdot)\dx=0
\] \\
for any $\phi \in C_c^{\infty}([0,T)\times \mathbb{T}^3)$;
}

\item {the momentum equation
\[
\int_0^T\int_{\mathbb{T}^3}
\Big(
(R+Q)\vu \cdot \p_t \vc{\varphi} + (R+Q) \vu\otimes \vu : \Grad_x \vc{\varphi}  +p(Z)\Div_x \vc{\varphi}
\Big) \dxdt
\]
\[
=
\int_0^T\int_{\mathbb{T}^3} \Big(\mu \Grad_x \vu : \Grad_x \vc{\varphi} +(\mu+\lambda)\Div_x\vu\, \Div_x \vc{\varphi} \Big)
 \dxdt-
\int_{\mathbb{T}^3} (R_0+Q_0)\vu_{0}\cdot \vc{\varphi} (0,\cdot)\dx
\] \\
for any $\vc{\varphi} \in C_c^{\infty}([0,T)\times \mathbb{T}^3;\R^3)$;
}

\item {the energy inequality holds a.e. in $(0,T)$
\[
\int_{\mathbb{T}^3}\left[
\f{1}{2}(R+Q)|\vu|^2+\f{1}{\gamma_{+}-1}\left( \f{R}{\alpha} \right)^{\gamma_{+}}\alpha
+\f{1}{\gamma_{-}-1}\left( \f{Q}{1-\alpha} \right)^{\gamma_{-}}(1-\alpha) \right](t,x)
\dx
\]
\[
+\int_0^t \int_{ \mathbb{T}^3 } \Big[ \mu |\Grad_x \vu|^2 +(\mu+\lambda)(\Div_x \vu)^2 \Big] \dxdt
\]
\[
\leq
\int_{\mathbb{T}^3}\left[
\f{1}{2}(R_0+Q_0)|\vu_0|^2+\f{1}{\gamma_{+}-1}\left( \f{R_0}{\alpha_0} \right)^{\gamma_{+}}\alpha_0
+\f{1}{\gamma_{-}-1}\left( \f{Q_0}{1-\alpha_0} \right)^{\gamma_{-}}(1-\alpha_0) \right]
\dx
\]
where we set
\[
\alpha:= \f{R}{Z}.
\]
}

\end{itemize}
\end{definition}

The existence of finite energy weak solutions for system (\ref{eq1}) was first proved by Bresch et al. \cite{BMZ19} in the semi-stationary regime, which was later extended by Novotn\'{y} et al. \cite{NM20} in the general case. In a series of work, for instance \cite{FJN,Ger11}, the fundamental property of weak-strong uniqueness was verified for the compressible Navier-Stokes system in the framework of weak solutions. However, due to the complicated form of pressure, much less is known for the compressible two-fluid models. Very recently, Jin et al. \cite{JN19} proved the weak-strong uniqueness for the two-fluid model of Baer-Nunziato type. It should be noticed that the basic tool in their proof is the celebrated relative energy inequality.

As a consequence, it is natural to explore the property of weak-strong uniqueness for the two-fluid model (\ref{eq1}). This is the motivation of the present note. However, due to the implicit form of pressure, it is not clear how to apply the well-developed method of relative entropy. Instead, we appeal to the Gronwall-type argument, inspired by Germain \cite{Ger11} and Desjardins \cite{Des97}. As the expense, the boundedness of densities for finite energy weak solutions is imposed additionally. More precisely, we have the following result.

\begin{theorem}\label{w-s}
Let $(R,Q,\vu)$ be a finite energy weak solution to (\ref{eq1})-(\ref{eq2}) such that
\[
(R,Q) \in L^{\infty} (0,T;L^{\infty}(\mathbb{T}^3)   ).
\]
Assume that $(\tilde{R},\tilde{Q},\tilde{\vu})$ is the classical solution to the same problem on $[0,T]$, starting from the same initial data. Then
\[
R=\tilde{R},\,\,Q=\tilde{Q},\,\,\vu=\tilde{\vu} \,\,\,\,  \text{    in    } [0,T]\times \mathbb{T}^3.
\]
\end{theorem}

Observe that the local existence and uniqueness of classical solutions as well as the global existence and uniqueness of classical solutions under smallness of initial data have recently been obtained by Piasecki and Zatorska \cite{PZ1} in a $L^p-L^q$ maximal regularity setting. Our main theorem gives the stability of classical solutions within finite energy weak solutions by imposing the boundedness of densities additionally. The rest of this note is devoted to its proof. 

\section{Weak-strong uniqueness principle}

In the sequel, we shall present the formal computations of the main steps without caring about the regularity issues of $(R,Q,\vu)$,  in analogy with the classical literature \cite{Des97,Ger11}. The rigorous proof is implemented with the standard regularization procedure, which is omitted here. 

To simplify the notations, we denote by
\[
     \mathfrak{R}=R- \tilde{R},\,\, \mathcal{Q}=Q- \tilde{Q}           ,\,\,\mathbf{U}:=\vu-\tilde{\vu},
\]
where $(R,Q,\vu)$ and $(\tilde{R},\tilde{Q},\tilde{\vu})$ satisfy the assumptions of Theorem \ref{w-s}. Moreover, due to the symmetry of the problem, we may assume, without loss of generality, that $\gamma_{+}\leq \gamma_{-}$, i.e., $\gamma\leq 1$.

To begin with, we estimate the $L^2(\mathbb{T}^3)$-norm of $(\mathfrak{R},\mathcal{Q})$.
\begin{lemma}\label{lem2.1}
\[
\f{\rm{d}}{\dt} \|\mathfrak{R}(t)\|_{L^2(\mathbb{T}^3)}
\leq
C \Big(
(\|R\|_{L^{\infty}(\mathbb{T}^3)}+ \|\tilde{R}\|_{L^{\infty}(\mathbb{T}^3)} ) \|\Grad_x\mathbf{U} \|_{L^2(\mathbb{T}^3)}
\Big)
\]
\beq\label{w-s-1}
+C\Big(
\|\Grad_x \tilde{\vu}\|_{L^{\infty}(\mathbb{T}^3)}\|\mathfrak{R}\|_{L^2(\mathbb{T}^3)}+
\| \Grad_x\tilde{R} \|_{L^3(\mathbb{T}^3)} \|\mathbf{U}\|_{L^6(\mathbb{T}^3)}
\Big);
\eeq

\[
\f{\rm{d}}{\dt} \|\mathcal{Q}(t)\|_{L^2(\mathbb{T}^3)}
\leq
C \Big(
(\|Q\|_{L^{\infty}(\mathbb{T}^3)}+ \|\tilde{Q}\|_{L^{\infty}(\mathbb{T}^3)} ) \|\Grad_x\mathbf{U} \|_{L^2(\mathbb{T}^3)}
\Big)
\]
\beq\label{w-s-2}
+C\Big(
\|\Grad_x \tilde{\vu}\|_{L^{\infty}(\mathbb{T}^3)}\|\mathcal{Q}\|_{L^2(\mathbb{T}^3)}+
\| \Grad_x\tilde{Q} \|_{L^3(\mathbb{T}^3)} \|\mathbf{U}\|_{L^6(\mathbb{T}^3)}
\Big).
\eeq

\end{lemma}

\begin{proof}
It follows from the continuity equations of $(R,\vu)$ and $(\tilde{R},\tilde{\vu})$ that
\[
\p_t \mathfrak{R} +\Div_x \Big( R\mathbf{U}+ \mathfrak{R}  \tilde{\vu}\Big)=0.
\]
Equivalently,
\[
\p_t \mathfrak{R}+R\Div_x \mathbf{U}+\mathbf{U}\cdot \Grad_x \mathfrak{R}
+\mathfrak{R} \Div_x \tilde{\vu}+\tilde{\vu}\cdot \Grad_x \mathfrak{R}
+\mathbf{U}\cdot \Grad_x \tilde{R}=0.
\]
Multiplying the above equation by $\mathfrak{R}$ and integrating over $\mathbb{T}^3$, with the help of H\"{o}lder's inequality, gives (\ref{w-s-1}); the verification of (\ref{w-s-2}) follows exactly the same way.  
\end{proof}

Next, we give the estimate of $\mathbf{U}$.
\begin{lemma}\label{lem2.2}
Suppose
\beq\label{w-s-3}
\|(R,Q,\tilde{R},\tilde{Q})\|_{  L^{\infty} (0,T;L^{\infty}(\mathbb{T}^3)   )  } \leq M.
\eeq
Then there exists a positive constant $C=C(M)$ such that
\[
\f{1}{2}\f{\rm{d}}{\dt}\| \sqrt{R+Q}\mathbf{U}    \|_{L^2(\mathbb{T}^3)}^2 +\| \Grad_x \mathbf{U} \|_{L^2(\mathbb{T}^3)}^2
\]
\[
\leq
(\| \mathfrak{R} \|_{L^2(\mathbb{T}^3)}+\|\mathcal{Q}\|_{L^2(\mathbb{T}^3)})  \|\p_t \tilde{\vu}+\tilde{\vu}\cdot \Grad_x \tilde{\vu}  \|_{L^3(\mathbb{T}^3)} \|  \mathbf{U}\|_{L^6(\mathbb{T}^3)}
\]
\beq\label{w-s-4}
+C \| \Grad_x \mathbf{U} \|_{L^2(\mathbb{T}^3)}  (\| \mathfrak{R} \|_{L^2(\mathbb{T}^3)}+\|\mathcal{Q}\|_{L^2(\mathbb{T}^3)})
+\| \sqrt{R+Q}\mathbf{U}    \|_{L^2(\mathbb{T}^3)}^2 \| \Grad_x \tilde{\vu}  \|_{L^{\infty}(\mathbb{T}^3)}.
\eeq
\end{lemma}
\begin{proof} The momentum equations easily imply
\[
(R+Q)\p_t \mathbf{U} + (R+Q)\vu \cdot \Grad_x \mathbf{U}+\Grad_x Z^{ \gamma_{+} }- \Grad_x \tilde{Z}^{ \gamma_{+} }
\]
\beq\label{w-s-5}
=\mu \Delta(\vu-\tilde{\vu})  +(\mu+\lambda)\Grad_x \Div_x (\vu-\tilde{\vu})
-(\mathfrak{R}+\mathcal{Q})( \p_t \tilde{u}+\tilde{u}  \cdot \Grad_x  \tilde{u} )-(R+Q)\mathbf{U}\cdot \Grad_x\tilde{\vu},
\eeq
where $\tilde{Z}$ is uniquely solved by
\begin{equation*}
\left\{\begin{aligned}
& \tilde{Q}=\left(1-\f{  \tilde{ R} }{  \tilde{Z}  }  \right)  \tilde{Z}^{\gamma},\\
& \tilde{R} \leq \tilde{Z}.\\
\end{aligned}\right.
\end{equation*}
Direct computations show
\[
\p_{R}Z= \f{  Z^{\gamma-1}  }{ \gamma Z^{\gamma-1} -(\gamma-1)R  Z^{\gamma-2}   },\,\,\p_{Q}Z= \f{  1 }{ \gamma Z^{\gamma-1} -(\gamma-1)R  Z^{\gamma-2}   };
\]
whence
\beq\label{w-s-6}
|\p_{R}Z  |\leq \f{1}{\gamma},\,\,|\p_{Q}Z  | \leq \f{ Z^{1-\gamma}   }{\gamma}.
\eeq
In addition, it is known that (see for instance Remark 1.1 in \cite{LSZ1})
\[
Z\leq \max \left\{ 2 \|R\| _{  L^{\infty} (0,T;L^{\infty}(\mathbb{T}^3)   )  },(2 \|Q\| _{  L^{\infty} (0,T;L^{\infty}(\mathbb{T}^3)   )  } )^{1/\gamma}  \right\};
\]
As a direct consequence of (\ref{w-s-3}),
\beq\label{w-s-7}
\|(Z,\tilde{Z})\|_{  L^{\infty} (0,T;L^{\infty}(\mathbb{T}^3)   )  } \leq C(M).
\eeq
We then deduce from (\ref{w-s-6})-(\ref{w-s-7}) that
\[
\left|
\int_{\mathbb{T}^3  }
\left( \Grad_x Z^{ \gamma_{+} }- \Grad_x \tilde{Z}^{ \gamma_{+} }   \right) \cdot \mathbf{U} \dx
\right|
\]
\[
\leq
\| Z^{ \gamma_{+} }-\tilde{Z}^{ \gamma_{+} } \|_{L^2(\mathbb{T}^3)} \| \Grad_x \mathbf{U} \|_{L^2(\mathbb{T}^3)}
\]
\beq\label{w-s-8}
\leq C(M)(\| \mathfrak{R} \|_{L^2(\mathbb{T}^3)}+\|\mathcal{Q}\|_{L^2(\mathbb{T}^3)}) \| \Grad_x \mathbf{U} \|_{L^2(\mathbb{T}^3)}.
\eeq

Testing (\ref{w-s-5}) by $\mathbf{U}$, with the help of (\ref{w-s-8}) and H\"{o}lder's inequality, and integrating by parts gives rise to (\ref{w-s-4}) immediately.  

\end{proof}

In order to apply the generalized Poincar\'{e} inequality, we give the estimate on the mean value of $\mathbf{U}$.
\begin{lemma}\label{lem2.3}
\[
\left|
\int_{\mathbb{T}^3 } \mathbf{U} \dx
\right|  \leq
\f{C}{  \int_{\mathbb{T}^3 } (R_0+Q_0) \dx   }\Big[
\left(
\| R\|_{L^{\infty}(\mathbb{T}^3)}+\|Q\|_{L^{\infty}(\mathbb{T}^3)}
\right)  \| \Grad_x \mathbf{U}\|_{L^2(\mathbb{T}^3)}
\]
\beq\label{w-s-9}
 +  \| \Grad_x \tilde{\vu}\|_{L^2(\mathbb{T}^3)} \left(\| \mathfrak{R} \|_{L^2(\mathbb{T}^3)}+\|\mathcal{Q}\|_{L^2(\mathbb{T}^3)}\right) \Big].
\eeq
\end{lemma}
\begin{proof} Obviously,
\[
\int_{\mathbb{T}^3} (R+Q) \mathbf{U}\dx=
\int_{\mathbb{T}^3} \left[ (R+Q) \left(\mathbf{U} -\int_{\mathbb{T}^3} \mathbf{U}\dx     \right)    +(R+Q)  \int_{\mathbb{T}^3} \mathbf{U}\dx \right]         \dx
\]
\beq\label{w-s-10}
=
\int_{\mathbb{T}^3}  (R+Q) \left(\mathbf{U} -\int_{\mathbb{T}^3} \mathbf{U}\dx     \right) \dx
+ \int_{\mathbb{T}^3} \mathbf{U}\dx  \int_{\mathbb{T}^3} (R_0+Q_0)\dx,
\eeq
since we know from the continuity equations that
\[
\int_{\mathbb{T}^3} (R+Q)\dx=\int_{\mathbb{T}^3} (R_0+Q_0)\dx.
\]
Thus,
\beq\label{w-s-11}
\int_{\mathbb{T}^3} \mathbf{U}\dx=-\f{1}{\int_{\mathbb{T}^3 } (R_0+Q_0) \dx }
\left[
\int_{\mathbb{T}^3}  (R+Q) \left(\mathbf{U} -\int_{\mathbb{T}^3} \mathbf{U}\dx     \right) \dx
-\int_{\mathbb{T}^3} (R+Q) \mathbf{U}\dx
\right].
\eeq
Observe next that
\[
\int_{\mathbb{T}^3} (R+Q) \mathbf{U}     \dx
=\int_{\mathbb{T}^3} (R+Q) (  \vu-\tilde{\vu} ) \dx
\]
\[
=\int_{\mathbb{T}^3} (R+Q)  \vu \dx-\int_{\mathbb{T}^3} (R+Q) \tilde{\vu}  \dx
\]
\[
=\int_{\mathbb{T}^3} (R_0+Q_0)  \vu_0 \dx-\int_{\mathbb{T}^3} ( \mathfrak{R} + \mathcal{Q} ) \tilde{\vu}  \dx -\int_{\mathbb{T}^3} (\tilde{R}+\tilde{Q}) \tilde{\vu}  \dx
\]
\beq\label{w-s-12}
=-\int_{\mathbb{T}^3} ( \mathfrak{R} + \mathcal{Q} ) \tilde{\vu}  \dx
=-\int_{\mathbb{T}^3} ( \mathfrak{R} + \mathcal{Q} )\left( \tilde{\vu}  -\int_{\mathbb{T}^3} \tilde{\vu}     \dx  \right)       \dx
\eeq
due to the fact that
\[
\int_{\mathbb{T}^3} ( \mathfrak{R} + \mathcal{Q} )  \dx=0.
\]
With the aid of H\"{o}lder and generalized Poincar\'{e} inequalities, (\ref{w-s-9}) follows from (\ref{w-s-11})-(\ref{w-s-12}) readily. The proof of Lemma \ref{lem2.3} is thus finished.  

\end{proof}

With Lemmas \ref{lem2.1}-\ref{lem2.3} at hand, we are now ready to give the proof of weak-strong uniqueness principle.
\begin{proof}[Proof of Theorem \ref{w-s}]
To begin with, we conclude from Lemma \ref{lem2.3} and Sobolev's inequality that
\beq\label{w-s-13}
\| \mathbf{U} \|_{L^6(\mathbb{T}^3)}
\leq
C \Big(  \| \Grad_x \mathbf{U}\|_{L^2(\mathbb{T}^3)}  +  \| \Grad_x \tilde{\vu}\|_{L^2(\mathbb{T}^3)} \left(\| \mathfrak{R} \|_{L^2(\mathbb{T}^3)}+\|\mathcal{Q}\|_{L^2(\mathbb{T}^3)}\right)       \Big).
\eeq
Combining (\ref{w-s-1}), (\ref{w-s-2}) and (\ref{w-s-13}), it follows that
\[
\f{\text{d}}{\dt}  \left( \|\mathfrak{R}(t)\|_{L^2(\mathbb{T}^3)}+\|\mathcal{Q}(t)\|_{L^2(\mathbb{T}^3)} \right)
\]
\beq\label{w-s-14}
\leq
C \Big[
\| \Grad_x \mathbf{U}\|_{L^2(\mathbb{T}^3)}
+\left(\|\Grad_x \tilde{\vu}\|_{L^{\infty}(\mathbb{T}^3)}+ \| \Grad_x \tilde{\vu}\|_{L^2(\mathbb{T}^3)} \right)
\left( \|\mathfrak{R}\|_{L^2(\mathbb{T}^3)}+\|\mathcal{Q}\|_{L^2(\mathbb{T}^3)} \right)
\Big].
\eeq
Upon invoking the classical Gronwall's inequality, the above differential inequality yields
\beq\label{w-s-15}
 \|\mathfrak{R}(t)\|_{L^2(\mathbb{T}^3)}+\|\mathcal{Q}(t)\|_{L^2(\mathbb{T}^3)}
\leq C  \int_0^t  \| \Grad_x \mathbf{U}(s)\|_{L^2(\mathbb{T}^3)} \ds.
\eeq
Consequently, the energy inequality (\ref{w-s-4}) may be strengthened as, with the help of (\ref{w-s-13}), (\ref{w-s-15}) and H\"{o}lder's inequality,
\[
\f{1}{2}\f{\text{d}}{\dt}\| \sqrt{R+Q}\mathbf{U}    \|_{L^2(\mathbb{T}^3)}^2 +\| \Grad_x \mathbf{U} \|_{L^2(\mathbb{T}^3)}^2
\]
\[
\leq C \left[
 \|\p_t \tilde{\vu}+\tilde{\vu}\cdot \Grad_x \tilde{\vu}  \|_{L^3(\mathbb{T}^3)} \| \Grad_x \mathbf{U}\|_{L^2(\mathbb{T}^3)} \int_0^t
 \| \Grad_x \mathbf{U}(s)\|_{L^2(\mathbb{T}^3)} \ds
\right]
\]
\[
+ C \left[
\|\p_t \tilde{\vu}+\tilde{\vu}\cdot \Grad_x \tilde{\vu}  \|_{L^3(\mathbb{T}^3)} \| \Grad_x \tilde{\vu}\|_{L^2(\mathbb{T}^3)}
 \left(\int_0^t  \| \Grad_x \mathbf{U}(s)\|_{L^2(\mathbb{T}^3)} \ds\right)^2
\right]
\]
\[
+C \left[
\| \Grad_x \mathbf{U}\|_{L^2(\mathbb{T}^3)}
\int_0^t
 \| \Grad_x \mathbf{U}(s)\|_{L^2(\mathbb{T}^3)} \ds
+\| \sqrt{R+Q}\mathbf{U}    \|_{L^2(\mathbb{T}^3)}^2 \| \Grad_x \tilde{\vu}  \|_{L^{\infty}(\mathbb{T}^3)}
\right]
\]
\[
\leq
C \left[
 \left(\|\p_t \tilde{\vu}+\tilde{\vu}\cdot \Grad_x \tilde{\vu}  \|_{L^3(\mathbb{T}^3)} +1\right) \| \Grad_x \mathbf{U}\|_{L^2(\mathbb{T}^3)} \int_0^t
 \| \Grad_x \mathbf{U}(s)\|_{L^2(\mathbb{T}^3)} \ds
\right]
\]
\[
+C\Big[
t  \|\p_t \tilde{\vu}+\tilde{\vu}\cdot \Grad_x \tilde{\vu}  \|_{L^3(\mathbb{T}^3)}\| \Grad_x \tilde{\vu}\|_{L^2(\mathbb{T}^3)}
\int_0^t
 \| \Grad_x \mathbf{U}(s)\|_{L^2(\mathbb{T}^3)}^2 \ds
\]
\beq\label{w-s-16}
 +\| \sqrt{R+Q}\mathbf{U}    \|_{L^2(\mathbb{T}^3)}^2 \| \Grad_x \tilde{\vu}  \|_{L^{\infty}(\mathbb{T}^3)}
\Big].
\eeq
By choosing
\[
f(t)=\f{1}{2}\| \sqrt{R+Q}\mathbf{U} (t)   \|_{L^2(\mathbb{T}^3)}^2+\f{1}{2} \int_0^t
 \| \Grad_x \mathbf{U}(s)\|_{L^2(\mathbb{T}^3)}^2 \ds,
\]
\[
g(t)=\int_0^t
 \| \Grad_x \mathbf{U}(s)\|_{L^2(\mathbb{T}^3)} \ds,
\]
\[
\alpha(t)=C \Big(
t  \|(\p_t \tilde{\vu}+\tilde{\vu}\cdot \Grad_x \tilde{\vu})(t)  \|_{L^3(\mathbb{T}^3)}\| \Grad_x \tilde{\vu}(t)\|_{L^2(\mathbb{T}^3)}
+\| \Grad_x \tilde{\vu}(t)  \|_{L^{\infty}(\mathbb{T}^3)}
\Big),
\]
\[
\beta(t)=C\Big(
\|(\p_t \tilde{\vu}+\tilde{\vu}\cdot \Grad_x \tilde{\vu})(t)  \|_{L^3(\mathbb{T}^3)}+1
\Big),
\]
in the generalized Gronwall's inequality (recalled in Section \ref{app-1}), we conclude from (\ref{w-s-16}) that
\[
\mathbf{U}=\mathbf{0},\,\, \mathfrak{R}=0,\,\,\mathcal{Q}=0,
\]
thus finishing the proof of Theorem \ref{w-s}.                

\end{proof}

{\large{\bf{Conflict of interest}}} The authors declare that there is no conflict of interest.

\section{An auxiliary lemma}\label{app-1}
For the convenience of the reader, we recall the following generalized Gronwall's inequality proved in Lemma 2.2 of Ref. \cite{Ger08}.
\begin{lemma}
Assume that
\[
f' +(g')^2\leq \alpha f +\beta g g',
\]
where $f,g'\alpha,\beta$ are positive functions with variable $t\in (0,T)$ and
\[
f\in L^{\infty}( (0,T) ),\,\, g(0)=0,\,\, g'\in L^{2}( (0,T) ),\,\,\alpha\in L^{1}( (0,T) ),\,\,\sqrt{t}\beta(t)\in L^{2}( (0,T) ).
\]
Then, for any $t\in [0,T]$, it holds
\[
e^{ -\int_0^t \alpha(\tau)\text{d} \tau   } f(t)
+ \left(
e^{ -\int_0^t \alpha(\tau)\text{d} \tau   } -\f{1}{2}-\f{1}{2} \int_0^t \tau \beta^2(\tau)d\tau
\right)\int_0^t (g'(\tau))^2 d\tau \leq f(0).
\]
\end{lemma}

\bibliographystyle{amsalpha}

\begin{thebibliography}{A}

\bibitem{BMZ19}Bresch, D., Mucha, P.B., Zatorska, E.: {Finite-energy solutions for compressible two-fluid Stokes system.} {Arch. Rational Mech. Anal.} {\bf 232}, 987-1029(2019)



\bibitem{BDGG}Bresch, D., Desjardins, B., Ghidaglia, J.M., Grenier, E., Hilliairet, M.: Multifluid models
including compressible fluids. Handbook of Mathematical Analysis in Mechanics of
Viscous Fluids, Eds. Y. Giga et A. Novotn\'{y} (2018), pp. 52.


\bibitem{Des97}Desjardins, B.: Regularity of weak solutions of the compressible isentropic Navier-Stokes equations 
{Comm. Partial Differ. Equ.} {\bf 22}, 977-1008(1997)





\bibitem{FJN}Feireisl, E., Bum Ja, J., Novotn\'{y}, A.: {Relative entropies, suitable weak solutions, and weak-strong uniqueness for the compressible Navier-Stokes system.} {J. Math. Fluid Mech.} {\bf 14}, 717-730(2012)



\bibitem{Ger08}Germain, P.: {Strong solutions and weak-strong uniqueness for the nonhomogeneous Navier-Stokes system.} {J. Anal. Math.} {\bf 105}, 169-196(2008)




\bibitem{Ger11}Germain, P.: {Weak-strong uniqueness for the isentropic compressible Navier-Stokes system.} {J. Math. Fluid Mech.} {\bf13}, 137-146(2011)





\bibitem{IsHi06}Ishii, M., Hibiki, T.: Thermo-Fluid Dynamics of Two-Phase Flow. Springer(2006)






\bibitem{JN19}Jin, B.J., Novotn\'{y}, A.: {Weak-strong uniqueness for a bi-fluid model for a mixture of non-interacting compressible fluids.} {J. Differential Equations.} {\bf268}, 204-238(2019)





%\bibitem{LS}Lions, P.L.: {Mathematical Topics in Fluid Mechanics, Vol. 2, Compressible Models}. Clarendon Press, Oxford, 1998


%\bibitem{LiMa98}P. L. Lions, N. Masmoudi, {Incompressible limit for a viscous compressible fluid.} {J. Math. Pures Appl. (9).} {\bf 77}, 585-627(1998)



\bibitem{LSZ1}Li, Y., Sun, Y., Zatorska, E.: {Large time behavior for a compressible two-fluid model with algebraic pressure closure and large initial data.} {Nonlinearity.} {\bf 33}, 4075-4094(2020)


%\bibitem{Mas1}N. Masmoudi, {Incompressible, inviscid limit of the compressible Navier-Stokes system.} {Ann. Inst. H. Poincar\'{e} Anal. Non Lin\'{e}aire.} {\bf 18}, 199-224(2001)




\bibitem{NM20}Novotn\'{y}, A., Pokorn\'{y}, M.: {Weak solutions for some compressible multicomponent fluid models.} {Arch. Rational Mech. Anal.} {\bf235}, 355-403(2020)



\bibitem{PZ1} Piasecki, T., Zatorska, E.: {Maximal regularity for compressible two-fluid system.}  {arXiv preprint: 2110.06584}

%\bibitem{Tem02}R. Temam, X. Wang, {Boundary layers associated with incompressible Navier-Stokes equations: the noncharacteristic boundary case.} {J. Differential Equations.} {\bf 179}, 647-686(2002)




%\bibitem{WJ06}S. Wang, S. Jiang, {The convergence of the Navier-Stokes-Poisson system to the incompressible Euler equations.} {Comm. Partial Differential Equations.} {\bf 31}, 571-591(2006)








\end{thebibliography}

\end{document}